\documentclass{article}

\usepackage{geometry}                % See geometry.pdf to learn the layout options. There are lots.
\usepackage{graphicx}
\usepackage{amssymb, mathrsfs, enumerate}

\usepackage{amscd}
\usepackage{amssymb}
\usepackage{graphics}
\usepackage{amsmath}
\usepackage[english]{babel}
\usepackage{hyperref}
\usepackage[latin1]{inputenc}

\usepackage{color}      % Need the color package
\usepackage{srcltx}  % more source specials

\usepackage{amsthm, amssymb}
\usepackage{setspace}
\usepackage[all]{xy}
\usepackage{graphicx}
\usepackage{ifpdf}
\usepackage{babel}
\usepackage{verbatim}

\newcommand{\BZ}{{\mathbb{Z}}}
\newcommand{\BN}{{\mathbb{N}}}
\newcommand{\BR}{{\mathbb{R}}}
\newcommand{\BC}{{\mathbb{C}}}

\newcommand{\BQ}{{\mathbb{Q}}}

\newcommand{\BG}{{\mathbb{G}}}

\newcommand{\BH}{{\mathbb{H}}}

\newcommand{\QQ}{{\mathbf{Q}}}
\newcommand{\RR}{{\mathbf{R}}}

\newcommand{\gd}{\delta}

\newcommand{\gC}{\Gamma}

\newcommand{\gs}{\sigma}
\newcommand{\gS}{\Sigma}
\newcommand{\gO}{\Omega}

\newcommand{\gep}{\epsilon}

\newcommand{\gL}{\Lambda}

\newcommand{\Ad}{\mathrm{Ad}}

\newcommand{\vol}{\mathrm{vol}}

\newcommand{\SL}{\mathrm{SL}}
\newcommand{\GL}{\mathrm{GL}}
\newcommand{\PSL}{\mathrm{PSL}}

\newcommand{\SO}{\mathrm{SO}}

\newcommand{\irs}{\mathrm{IRS}}
\newcommand{\sub}{\mathrm{Sub}}

\newcommand{\act}{\curvearrowright}

\newtheorem{prop}{Proposition}[section]
\newtheorem{thm}[prop]{Theorem}
\newtheorem{lem}[prop]{Lemma}
\newtheorem{cor}[prop]{Corollary}
\newtheorem{conj}[prop]{Conjecture}
\newtheorem*{conj*}{Conjecture}
\theoremstyle{definition}

\newtheorem{defn}[prop]{Definition}
\newtheorem{rem}[prop]{Remark}
\newtheorem*{rem*}{Remark}

\newtheorem{exam}[prop]{Example}
\newtheorem{prob}[prop]{Problem}
\newtheorem{ques}[prop]{Question}

\newtheorem{clm}[prop]{Claim}

\newtheorem{theorem}{Theorem}
\newtheorem{acknowledgement}[theorem]{Acknowledgement}

\newtheorem{corollary}[theorem]{Corollary}

\newtheorem{definition}[theorem]{Definition}

\newcommand{\N}{{\mathbf{N}}}

\newcommand{\R}{{\mathbf{R}}}

\newcommand{\Sub}{\operatorname{Sub}}
\newcommand{\Cay}{\operatorname{Cay}}
\newcommand{\Spec}{\operatorname{Spec}}
\newcommand{\inj}{\operatorname{Inj}}
\newcommand{\IRS}{\operatorname{IRS}}

\title{A view on Invariant Random Subgroups and Lattices}
\author{Tsachik Gelander}
\begin{document}
\maketitle

\begin{abstract}
For more than half a century  
lattices in Lie groups played an important role in geometry, number theory and group theory.
Recently the notion of  
Invariant Random Subgroups (IRS) 
emerged as a natural generalization of lattices. It is thus intriguing to extend results from the theory of lattices to the context of IRS, and to study lattices by analyzing the compact space of all IRS of a given group. This article focuses on the interplay between lattices and IRS, mainly in the classical case of semisimple analytic groups over local fields.
\end{abstract}

%\section{introduction}

Let $G$ be a locally compact group. We denote by $\sub(G)$ the space of closed subgroups of $G$ equipped with the Chabauty topology. 
The compact space $\sub(G)$ is usually too complicated to work with directly. However, considering a random point in $\sub(G)$ is often much more effective.
Note that $G$ acts on $\sub(G)$ by conjugation. An invariant random subgroup (or shortly IRS) is a $G$-invariant probability measure on $\sub(G)$. We denote by $\irs(G)$ the space of all IRSs of $G$ equipped with the $w^*$-topology. By Riesz' representation theorem and Alaoglu's theorem, $\irs(G)$ is compact. 

The Dirac measures in $\irs(G)$ correspond to normal subgroups. Any lattice $\gC$ in $G$ induces an IRS $\mu_\gC$ which is defined as the push forward of the $G$-invariant probability measure from $G/\gC$ to $\sub(G)$ via the map $g\gC\mapsto g\gC g^{-1}$. 

More generally consider a probability measure preserving action  $G\act (X,m)$. By a result of Varadarajan, the stabilizer of almost every point in $X$ is closed in $G$. Moreover, the stabilizer map $X\to \text{Sub}_G,~x\mapsto G_x$ 
is measurable, and hence one can push the measure $m$ to an IRS on $G$. In other words the random subgroup is the stabilizer of a random point in $X$.
%This reflects the connection between invariant random subgroups and pmp actions. Moreover, 
In a sense, the study of pmp $G$-spaces can be divided to 
the study of stabilizers (i.e. IRSs), 
the study of orbit spaces
and the interplay between the two. 
Vice versa, every IRS arises (non-uniquely) in this way (see \cite[Theorem 2.6]{7S}). 
%Moreover, every IRS in a locally compact group $G$ is the stabilizer of some pmp action (\cite[Theorem 2.6]{7S}).
%This was proved in \cite{AGV} for discrete groups and in \cite[Theorem 2.4]{7S} for general $G$. The first thing that comes to mind is to take the given $G$ action on $(\text{Sub}_G,\mu)$, but then the stabilizer of a point $H\in\text{Sub}_G$ is $N_G(H)$ rather than $H$. To correct this one can consider the larger space $\text{Cos}_G$ of all cosets of all closed subgroups, as a measurable $G$-bundle over $\text{Sub}_G$. Defining an appropriate invariant measure on $\text{Cos}_G\times\BR$ and replacing each fiber by a Poisson process on it, gives the desired probability space.
%\medskip

%\tableofcontents

%Since it has been reborn 
Since its rebirth in the beginning of the current decade (see \S \ref{sec:history} for a short summery of the history of IRS), the topic of IRS played an important role in various parts of group theory, geometry and dynamics,  and attracted the attention of many mathematicians for various different reasons. Not aiming to give an overview, I will try to highlight here several aspects of the evolving theory. 

%%%%%%%%%%%%%%%%%%

\section{IRS and Lattices}
IRSs can be considered as a generalisation of lattices, and one is tempted to extend results from the theory of lattices to IRS. In the other direction, as often happens in mathematics where one considers random objects to prove result about deterministic ones, the notion of IRS turns out to yield an extremely powerful tool to study lattices. In this section I will try to give a taste of the interplay between IRSs and Lattices focusing mainly on the second point of view. 
Attempting to expose the phenomenon in a rather clear way, avoiding technicality, I will assume throughout most of this section that $G$ is a noncompact simple Lie group, although most of the results can be formulated in the much wider setup of semisimple analytic groups over arbitrary local fields, see \S \ref{par:general-versions}.

\subsection{Borel Density}
Let $\text{P}\sub(G)$ denote the space of proper closed subgroups. Since $G$ is an isolated point in $\sub(G)$ (see \cite{To,Ku}) we deduce that $\text{P}\sub(G)$ is compact. Letting $\text{PIRS}(G)$ denote the subspace of $\irs(G)$ consisting of the measures supported on $\text{P}\sub(G)$, we deduce:

\begin{lem}\label{lem:PSUB}
The space of proper IRSs, $\text{PIRS}(G)$ is compact.
\end{lem}

Let us say that that an IRS $\mu$ is discrete if a random subgroup is $\mu$ almost surely discrete, and denote by $\text{DIRS}(G)$ the subspace of $\irs(G)$ consisting of discrete IRSs. The following is a generalization of the classical Borel Density Theorem:

\begin{thm}\label{thm:BDT}(Borel Density Theorem for IRS, \cite[Theorem 2.9]{7S})
Every proper IRS in $G$ is discrete, i.e. $\text{PIRS}(G)=\text{DIRS}(G)$. Moreover, for every $\mu\in \text{DIRS}(G)$, $\mu$-almost every subgroup is either trivial or Zariski dense.
\end{thm}

In order to prove Theorem \ref{thm:BDT} one first observes that there are only countably many conjugacy classes of non-trivial finite subgroups in $G$, hence the measure of their union is zero with respect to any non-atomic IRS. Then one can apply the same idea as in Furstenberg's proof of the classical Borel density theorem \cite{Fur}. Indeed, taking the Lie algebra of $H\in\sub(G)$ as well as of its Zariski closure induce measurable maps (see \cite[\S 4]{non-archimedean})
$$
 H\mapsto \text{Lie}(H),~H\mapsto\text{Lie}(\overline{H}^Z).
$$
As G is noncompact, Furstenberg's argument implies that the Grassman variety of non-trivial subspaces of $\text{Lie}(G) $ does not carry an $\Ad(G)$-invariant measure. It follows that $ \text{Lie}(H)= 0$ and $\text{Lie}(\overline{H}^Z)\in\{
\text{Lie}(G),0\}$
almost surely, and the two statements of the theorem follow.

%and since a non-trivial Grassman variety associated to subspaces of Lie$(G)$ does not carry an Ad$(G)$-invariant measure (because $G$ is noncompact) one deduces that $\text{Lie}(H)=0$, i.e. that $H$ is discrete, and $\text{Lie}(\overline{H}^Z)=\text{Lie}(G)$ almost surely (see \cite{7} for more details).

%%%%%%%%%%%%%

\subsection{Weak Uniform Discreteness}

Let $U$ be an identity neighbourhood in $G$. A family of subgroups $\mathcal{F}\subset\sub(G)$ is called {\it $U$-uniformly discrete} if $\gC\cap U=\{1\}$ for all $\gC\in \mathcal{F}$.

\begin{defn}
A family $\mathcal{F}\subset\text{DIRS}(G)$ of invariant random subgroups is said to be {\it weakly uniformly discrete} if for every $\gep>0$ there is an identity neighbourhood $U_\gep\subset G$ such that
$$
 \mu (\{\gC\in\sub_G:\gC\cap U_\gep\ne\{1\}\})<\gep
$$
for every $\mu\in \mathcal{F}$.
\end{defn}

%The quality of this definition is demonstrated in the following result (from \cite{wud}) which yield a valuble information (in particular when restricting to lattices) while being proved by an elementary argument.

A justification for this definition is given by the following result which is proved by an elementary argument and yet provides a valuable information:

\begin{thm}\label{thm:wud}
Let $G$ be a connected non-compact simple Lie group. Then $\text{DIRS}(G)$ is weakly uniformly discrete.
\end{thm}

Let $U_n, n\in\BN$ be a descending sequence of compact sets in $G$ which form a base of identity neighbourhoods, and set
$$
  K_n=\{\gC\in\sub_G:\gC\cap U_n=\{ 1 \}\}.
$$

Since $G$ has NSS (no small subgroups), i.e. there is an identity neighbourhood which contains no non-trivial subgroups, we have:

\begin{lem}\label{K_n-open}
The sets $K_n$ are open in $\sub(G)$.
\end{lem}

\begin{proof}
Fix $n$ and
let $V\subset U_n$ be an open identity neighbourhood which contains no non-trivial subgroups, such that ${V^2}\subset U_n$. It follows that a subgroup $\gC$ intersects $U_n$ non-trivially iff it intersects $U_n\setminus V$. Since $U_n\setminus V$ is compact, the lemma is proved.  
\end{proof}

In addition, observe that the ascending union $\bigcup_n K_n$ exhausts $\sub_d(G)$, the set of all discrete subgroups of $G$. Therefore we have:

\begin{clm}\label{clm}
For every $\mu\in\text{DIRS}(G)$ and $\gep>0$ we have $\mu(K_n)>1-\gep$ for some $n$.\qed
\end{clm}

Let
$$
 \mathcal{K}_{n,\gep}:= \{\mu\in\text{DIRS}(G):\mu(K_n)>1-\gep\}.
$$
Since $\sub(G)$ is metrizable, it follows from Lemma \ref{K_n-open} that $\mathcal{K}_{n,\gep}$ is open. 
By Claim \ref{clm}, for any given $\gep>0$, the sets $\mathcal{K}_{n,\gep},~n\in\BN$ form an ascending cover of $\text{DIRS}(G)$. Since the latter is compact, we have $\text{DIRS}(G)\subset \mathcal{K}_{m,\gep}$ for some $m=m(\gep)$. 
It follows that 
$$
 \mu\big(\{ \gC\in\sub(G):\gC~\text{intersects}~U_m~\text{trivially}\}\big)>1-\gep,
$$
for every $\mu\in\text{DIRS}(G)$. Thus Theorem \ref{thm:wud} is proved.
\qed

\medskip

Picking $\gep<1$ and applying the theorem for the IRS $\mu_\gC$ where $\gC\le G$ is an arbitrary lattice, one deduces the Kazhdan--Margulis theorem \cite{KM}, and in particular that there is a positive lower bound on the volume of locally $G/K$-orbifolds:

\begin{cor}[Kazhdan--Margulis theorem]\label{cor:KM}
There is an identity neighbourhood $\gO\subset G$ such that for every lattice $\gC\le G$ there is $g\in G$ such that $g\gC g^{-1}\cap \gO=\{1\}$.
\end{cor} 

\begin{figure}[h]
    \centering
    \includegraphics[width=0.8\textwidth]{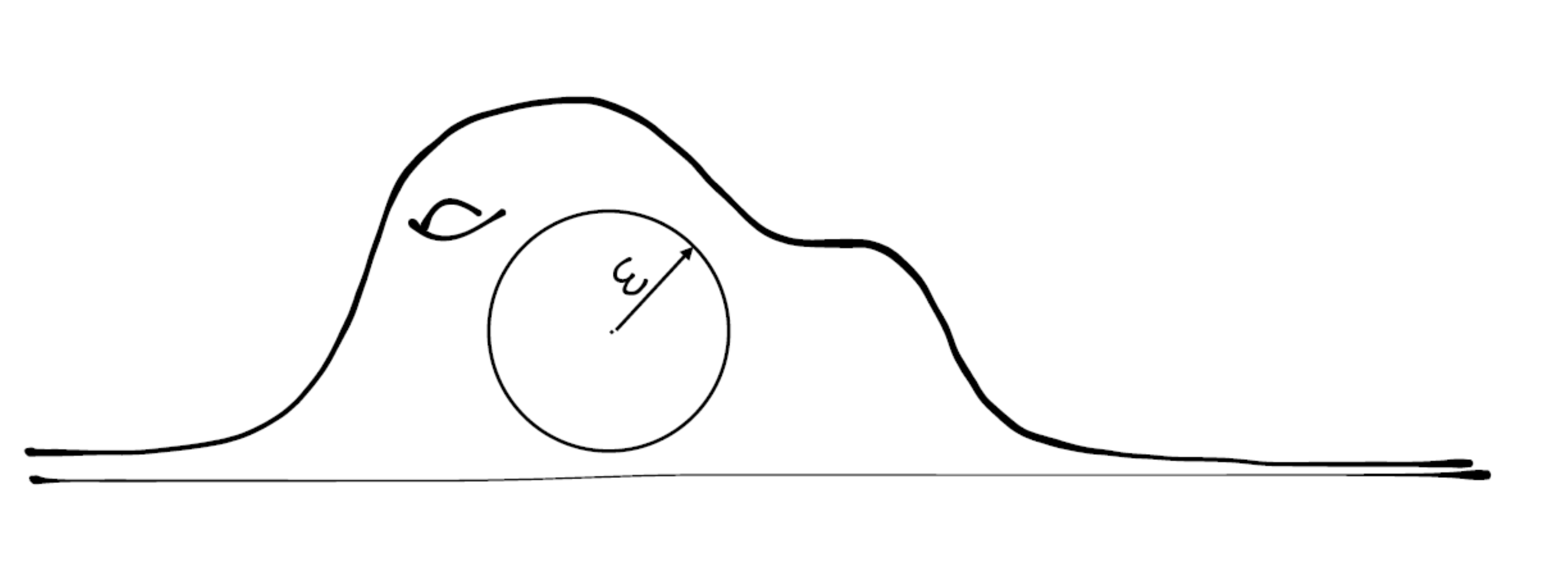}
    \caption{Every $X$-manifold has a thick part.}
    \label{fig:awesome_image}
\end{figure}

A famous conjecture of Margulis \cite[page 322]{Ma} asserts that the set of all torsion-free anisotropic arithmetic lattices in $G$ is $U$-uniformly discrete for some identity neighbourhood $U\subset G$. 
Theorem \ref{thm:wud} can be regarded as a probabilistic variant of this conjecture as it implies that all lattices in $G$ are jointly weakly uniformly discrete.

%Recall that every probability measure preserving $G$-space $(X,m)$ corresponds to an IRS, via the stabilizer map $X\ni x\mapsto G_x\in\sub(G)$ and that, vice versa, every IRS arises (non-uniquely) in this way (see \cite[Theorem 2.6]{7S}). Thus,
In the language of pmp actions
Theorem \ref{thm:wud} can be reformulated as follows:

\begin{thm}[p.m.p.\ actions are uniformly weakly locally free]\label{thm:p.m.p.}
For every $\gep>0$ there is an identity neighbourhood $U_\gep\subset G$ such that the stabilizers of $1-\gep$ of the points, in
any non-atomic probability measure preserving $G$-space $(X,m)$ are $U_\gep$-uniformly discrete. I.e.,
there is a subset $Y\subset X$ with $m(Y)>1-\gep$ such that $U_\gep\cap G_y=\{1\},~\forall y\in Y$.
% intersects trivially the stabilizer of every $y\in Y$.
%$u\cdot y\ne y$ for all $y\in Y$ and $u\in U$.
\end{thm}

%%%%%%%%%%%%%%%%

\subsection{Local Rigidity}

Observe that local rigidity implies Chabauty locally rigid:

\begin{prop}\label{thm:CLR}
Let $G$ be a connected Lie group and $\gC\le G$ a locally rigid lattice. Then $\gC$ is Chabauty locally rigid, i.e. the conjugacy class of $\Gamma$ is Chabauty open.
\end{prop}

\begin{proof}
Let $\gC\le G$ a locally rigid lattice.
Let $U$ be a compact identity neighborhood in $G$ satisfying:
\begin{itemize}
\item $U\cap\gC=\{1\}$,
\item $U$ contains no nontrivial groups, 
\end{itemize}
and let $V$ be an open symmetric identity neighborhood with $V^2\subset U$. By the choice of $V$ we for a subgroup $H\le G$, that $H\cap U\ne\{1\}$ iff $H$ meets the compact set $U\setminus V$.  

Recall that $\gC$, being a lattice in a Lie group, is finitely presented and let $\langle \gS|R\rangle$ be a finite presentation of $\gC$. Denote $S=\{s_1,\ldots,s_k\}$.
We can pick a sufficiently small identity neighborhood $\gO$ so that for every choice of $g_i\in s_i\gO,~i=1,\ldots,k$ and every $w\in R$ we have
$w(g_1,\ldots,g_n)\in U$. 

Now if $H\in\sub(G)$ is sufficiently close to $\gC$ in the Chabauty topology then $H\cap s_i\gO\ne\emptyset,~i=1,\ldots,k$ and $H\cap (U\setminus V)=\emptyset$, i.e. $H\cap U=\{1\}$. Picking $h_i\in H\cap s_i\gO,~i=1,\ldots,k$ one sees that the assignment $s_i\mapsto h_i$ induces a homomorphism from $\gC$ into $H$. Since $\gC$ is locally rigid it follows that if $H$ is sufficiently close to $\gC$ then it contains a conjugate of $\gC$. 
%The assumption that $H$ meets $U$ trivially guaranties that $H$ is discrete and its co-volume is bounded below by $\vol(V)$. 
However there are only finitely many subgroups containing $\gC$ and intersecting $U$ trivially, hence if $H$ is sufficiently close to $\gC$ then it is a conjugate of $\gC$.
\end{proof} 

Denote by $\text{EIRS}(G)$ the space ergodic IRSs of $G$, i.e. the set of extreme points of $\irs(G)$.

\begin{cor}
Let $G$ be a connected Lie group and $\gC\le G$ a locally rigid lattice. Then the IRS $\mu_\gL$ is isolated in $\text{EIRS}(G)$.
\end{cor}

\begin{proof}
Let $\gC$ be as above.
If $\mu$ is an IRS of $G$ sufficiently close to $\mu_\gC$ then with positive $\mu$-probability a random subgroup is a conjugate of $\gC$. Thus if $\mu$ is ergodic then it must be $\mu_\gC$.
\end{proof}

%%%%%%%%%%%%%%%

\subsection{Farber property}

\begin{defn}
A sequence $\mu_n$ of invariant random subgroups of $G$ is called {\it Farber}\footnote{Various authors use various variants of this notion.}
 if $\mu_n$ converge to the trivial IRS, $\gd_{\{1\}}$.
\end{defn}

One of the key results of \cite{7S} is the following theorem:

\begin{thm}\label{thm:7main}
Let $G$ be a simple Lie group of real rank at least $2$. Let $\gC_n$ be a sequence of pairwise non-conjugate lattices, then $\mu_{\gC_n}$ is Farber.
\end{thm}

The proof relies on the following variant of Stuck--Zimmer theorem (see \cite{7S}):

\begin{thm}
Assuming $\text{rank}(G)\ge 2$, a proper non-trivial ergodic IRS of $G$ is $\mu_\gC$ for some lattice $\gC$, i.e.
$$
 \text{EIRS}(G)=\{\gd_G,\gd_{\{1\}}, \mu_\gC,~\gC~\text{a lattice in}~G\}.
$$
\end{thm}

\begin{proof}[Proof of Theorem \ref{thm:7main}]
Since $G$ is of rank $\ge 2$ it has Kazhdan's property $(T)$, \cite{Ka}. By \cite{GW}, $\text{EIRS}(G)$ is compact. Since $\gd_G$ and $\gd_{\mu_\gC}$ where $\gC$ is a lattice in $G$ are isolated points, it follows that $\gd_{\{1\}}$ is the unique accumulation point of $\text{EIRS}(G)$. In particular $\mu_{\gC_n}\to\gd_{\{1\}}$.
\end{proof}

%One can reformulate the results above as follows:
%
%\begin{thm}
%The space $\text{EIRS}(G)$ is compact, consisting of the isolated points 
%$$
% \{\gd_G, \mu_\gC,~\gC~\text{a lattice in}~G\}
% $$ 
% and a unique accumulation point $\gd_{\{1\}}$.
%\end{thm}
%

%%%%%%%%%%%%%%%%%%%%%

\subsection{Semisimple analytic groups}\label{par:general-versions}

Above we have carried out the arguments under the assumption that $G$ is a simple Lie group. In this section we state the results in the more general setup of analytic groups over local fields.

\begin{definition}
\label{def:semisimple analytic group}
Let $k$ be a  local field and $\BG$ a connected  $k$-isotropic $k$-simple  linear $k$-algebraic group.
\begin{itemize}
\item A \emph{simple analytic group} is a group of the form $\BG(k)$.
\item A \emph{semisimple analytic group} is an almost direct product of finitely many simple analytic groups, possibly over different local fields.
\end{itemize}
\end{definition}

Note that if $k$ is a local field and $\BG$ is a connected semisimple linear $k$-algebraic group without $k$-anisotropic factors then $\BG(k)$ is a semisimple analytic group. Such a group is indeed \emph{analytic} in the sense of e.g. \cite{sere}.
%
%A semisimple analytic group is respectively \emph{non-Archimedean/has zero characteristic/simply connected} if all of its simple analytic almost direct factors are non-Archimedean/defined over a local field of zero characteristic/are simply connected.
%
Associated to a  semisimple analytic group $G$ are its universal covering group $\widetilde{G}$ and adjoint group $\overline{G}$. There are central $k$-isogenies $\widetilde{G} \xrightarrow{\widetilde{p}} G \xrightarrow{\overline{p}} \overline{G} $ and this data is unique up to a $k$-isomorphism \cite[I.4.11]{Ma}. 
%
%
%A semisimple analytic group is \emph{happy} if its non--Archimedean factor contains a topologically finitely generated compact open subgroup.  This notion  is motivated by \cite{barnea} and discussed in \S\ref{sub:G+ and happy} below. All semisimple analytic groups zero are happy in zero characteristic, and in positive characteristic all simply-connected semisimple analytic group are happy.
%
For a semisimple analytic group $G$, denote by $G^+$ the subgroup  of $G$ generated by its unipotent elements \cite[I.1.5,I.2.3]{Ma}.
If $G$ is simply connected then $G = G^+$. If $G$ is Archimedean then $G^+$ is the connected component $G_0$ at the identity. In general $G/G^+$ is a compact abelian group. The group $G^+$ admits no proper finite index subgroups.

\begin{definition}
\label{def:happy semisimple analytic group}
A simple analytic group $G$ is \emph{happy} if $\mathrm{char}(k)$ does not the divide $|Z|$ where $Z$ is the kernel of the map $\widetilde{G} \to G$. A semisimple analytic group is \emph{happy} if all of its almost   direct factors are.
\end{definition}

Note that a simply connected or a zero characteristic semisimple analytic group is automatically happy.
From the work of Barnea and Larsen \cite{BL} one obtains that a semisimple analytic group $G$ is happy, iff $G/G^+$ is a finite abelian group, iff some (equivalently every) compact open subgroup in the  non-Archimedean factor of $G$ is  finitely generated.

%For happy groups the subgroup $G^+$ is particularly nicely behaved.
%\blue{
%\begin{thm}[Barnea--Larsen \cite{BL}]
%	\label{thm:properties of happy groups}
%	The following are equivalent for a semisimple analytic group $G$. 
%	\begin{enumerate}
%\item \label{item:G happy} $G$ is happy,
%\item \label{item:map is open} The central $k$-isogeny $\widetilde{p} : \widetilde{G}\to G$ is separable, or equivalently is an open map in the Hausdorff topology,
%\item \label{item:quotient is finite} $G/G^+$ is a finite abelian group,
%\item \label{item:finitely generated compact open} Some (equivalently every) compact open subgroup in the  non-Archimedean factor of $G$ is  finitely generated.
%\end{enumerate}
%\end{thm}
%}
%The work of Barnea and Larsen focuses on property (\ref{item:finitely generated compact open}). Some of the equivalences between (\ref{item:G happy}), (\ref{item:map is open}) and (\ref{item:quotient is finite}) are  discussed already e.g. in  \cite{borel1973homomorphismes}.

\paragraph{Self Chabauty isolation }
\label{sec:isolated groups}
\begin{definition}
\label{def:isolated}
A l.c.s.c. group $G$ is \it{self-Chabauty-isolated} if the point $G$ is  isolated in $\sub(G)$ with the Chabauty topology.
\end{definition}

Note that $G$ is self-Chabauty-isolated if and only if there is a finite collection  of open subsets $U_1,\ldots,U_n \subset G$ so that the only closed subgroup intersecting  every $U_i$ non-trivially is $G$ itself. The following result is proved in \cite[\S 6]{non-archimedean}.

\begin{theorem}
\label{thm:G is isolated}
%Let $\GG$ be a connected simply connected semisimple linear $k$-group. Denote $G = \GG(k)$. Then 
Let $G$ be a happy semisimple analytic group. Then $G^+$ is self-Chabauty-isolated.
\end{theorem}

As an immediate consequence we deduce the analog of Lemma \ref{lem:PSUB}, namely that the space $\text{P}\sub(G^+)$ is compact for every $G$ as in Theorem \ref{thm:G is isolated}.

\paragraph{Borel Density.}

The following generalization of Theorem \ref{thm:BDT} was obtained in \cite[Theorem 1.9]{non-archimedean}:

\begin{theorem}[Borel density theorem for IRS]
\label{thm:borel density for IRS}
%Let $k$ be a non-Archimedean local field. Let $\GG$ be a connected, simply connected, semisimple linear $k$-group without $k$-anisotropic factors. Denote $ G = \GG(k)$.  
%Let $G$ be a semisimple analytic group.
Let $k$ be a local field and $G$ a happy semisimple analytic group over $k$. Assume that $G$ has no  almost $k$-simple factors of type $B_n,C_n$ or $F_4$ if $\mathrm{char}(k) = 2$ and of type $G_2$ if $\mathrm{char}(k) = 3$. 

%\marginpar{can this assumption on non-standard isogenies be removed?}

Let $\mu$ be an  ergodic invariant random subgroup of $G$. Then there is a pair of  normal subgroups $N,M \lhd G$ so that 
 $$ N \le H \le M,  \quad  \text{$H/N$ is discrete in $G/N$} \quad \text{and} \quad \overline{H}^{\mathrm{Z}} = M $$
 for $\mu$-almost every closed subgroup $H$ in $  G$.  Here $\overline{H}^{\mathrm{Z}}$ is the Zariski closure of $H$. 
\end{theorem}

\paragraph{Weak Uniform Discreteness.}

As shown in \cite[Theorem 2.1]{WUD} the analog of Theorem \ref{thm:wud} holds for general semisimple Lie groups:

\begin{thm}
Let $G$ be a connected center-free semisimple Lie group with no compact factors. Then $\text{DIRS}(G)$ is weakly uniformly discrete. 
\end{thm}

Consider now a general locally compact $\gs$-compact group $G$.
Since $\sub_d(G)\subset\sub(G)$ is a measurable subset, by restricting attention to it, one may replace Property NSS
%being the inverse image of the map $\psi$, introduced in the proof of Proposition \ref{prop:BD}, 
by the weaker Property NDSS (no discrete small subgroups), which means that there is an identity neighbourhood which contains no non-trivial discrete subgroups. In that generality, the analog of Lemma \ref{K_n-open} would say that $K_n$ are relatively open in $\sub_d(G)$. Thus, the ingredients required for the argument above are:

%The proof of Theorem \ref{thm:main} relies on the following ingredients:

\begin{enumerate}
\item $\text{DIRS}(G)$ is compact,
\item $G$ has NDSS.
\end{enumerate}

In particular we have:

\begin{thm}\label{thm:general}
Let $G$ be a locally compact $\gs$-compact group which satisfies $(1)$ and $(2)$. Then $\text{DIRS}(G)$ is weakly uniformly discrete.
\end{thm}

If $G$ possesses the Borel density theorem and $G$ is self-Chabauty-isolated then $(1)$ holds. By the previous paragraphs happy semisimple analytic groups enjoy these two properties, and hence $(1)$. 
%
%For a locally compact totally disconnected group $G$, the following are equivalent:
%\begin{itemize}
%\item $G$ has property NDSS.
%\item $G$ admits an open torsion-free subgroup.
%\item The space $\sub_d(G)$ is uniformly discrete.
%\end{itemize}

\paragraph{\bf $p$-adic groups.}
Note that a $p$-adic analytic group $G$ has NDSS,
%, hence, in view of the Borel density theorem (proved in \cite{GL}), if $G$ is an almost simple $p$-adic analytic group then 
and hence $\text{DIRS}(G)$ is uniformly discrete (in the obvious sense).
% satisfy a strong version of Theorem \ref{thm:general}, namely:
%
%\begin{thm}[See \cite{GL} for more details]
%Let $G$ be an almost simple $p$-adic analytic group. Then $\irs_d(G)$ is uniformly discrete.
%\end{thm}
Moreover, if $G\le\GL_n(\BQ_p)$ is a rational algebraic subgroup, then the first principal congruence subgroup $G(p\BZ_p)$ is a torsion-free open compact subgroup. 
In particular the space $\text{DIRS}(G)$ is $G(p\BZ_p)$-uniformly discrete.
Supposing further that $G$ is simple, then in view of the Borel density theorem we have:\\

{\it Let $(X,\mu)$ be a probability $G$-space essentially with no global fixed points.
Then the action of the congruence subgroup $G(p\BZ_p)$ on $X$ is essentially free.} 

\paragraph{\bf Positive characteristic.}
Algebraic groups over local fields of positive characteristic do not posses property NDSS, and the above argument does not apply to them. 

\begin{conj}\label{conj:wud}
Let $k$ be a local field of positive characteristic, let $\BG$ be simply connected absolutely almost simple $k$-group with positive $k$-rank 
and let $G=\BG(k)$ be the group of $k$-rational points. Then $\text{DIRS}(G)$ weakly uniformly discrete.
\end{conj}

The analog of Corollary \ref{cor:KM} in positive characteristic was proved in \cite{Ali,Rag}. 
A. Levit proved Conjecture \ref{conj:wud} for $k$-rank one groups \cite{Levit:wud}.

\paragraph{Local Rigidity}

Combining Theorem 7.2 and Proposition 7.9 of  \cite{non-archimedean} we obtain the following extension of Theorem \ref{thm:CLR}:

\begin{theorem}(Chabauty local rigidity \cite{non-archimedean})
\label{thm:Chabauty local rigidity for lattices in algebraic groups}
Let $G$ be a semisimple analytic group and $\Gamma$ an irreducible lattice in $G$. If $\gC$ is locally rigid then it is also Chabauty locally rigid.
%Let $k$ be a non-Archimedean local field, $\GG$ a connected simply connected semisimple $k$-algebraic group without $k$-anisotropic factors and $G = \GG(k)$. 
%If $\mathrm{rank}(G) \ge 2$ then $\Gamma$ admits a Chabauty neighborhood consisting of  conjugates of $\Gamma$ by automorphisms of $G$. 
%If  $G$ is  defined over zero characteristic then these automorphisms are inner.
\end{theorem}

Let us also mention the following generalization of the classical Weil local rigidity theorem:

\begin{theorem}(\cite[Theorem 1.2]{local-rigidity}
	\label{thm:local rigidity of uniform in general cat0-intro}
	%[Theorem \ref{thm:local rigidity of uniform in general cat0}]
	Let $X$ be a proper geodesically complete $\text{CAT}(0)$ space without Euclidean
	factors and with $\text{Isom}{X}$ acting cocompactly. 
	Let $\Gamma$ be a uniform lattice
	in $\text{Isom}{X}$. Assume that for every de Rham factor $Y$ of $X$ isometric
	to the hyperbolic plane the projection of $\Gamma$ to $\text{Isom}{Y}$
	is non-discrete. Then $\Gamma$ is locally rigid.
\end{theorem}

\paragraph{Farber Property}

The proof presented above for Theorem \ref{thm:7main} follows the lines developed at \cite{non-archimedean} and
is simpler and applies to a more general setup than the original proof from \cite{7S}. 
In particular, the following general version of \cite[Theorem 4.4]{7S} is proved in \cite{non-archimedean}:

\begin{theorem}\cite[Theorem 1.1]{non-archimedean}
	\label{thm:accumulation points of invariant random subgroups}
	%Let $G$ be a semisimple analytic group 
	Let $G$ be a semisimple analytic group. 
	Assume that $G$ is happy, has property $(T)$ and  $\mathrm{rank}(G) \ge 2$.
	Let $\Gamma_n$ be a sequence of pairwise non-conjugate irreducible lattices in $G$.  Then $\gC_n$ is Farber.
\end{theorem}

\paragraph{Farber property for congruence subgroups.}

%In certain situations it is possible to extend the  method of invariant random subgroups to obtain Benjamini---Schramm convergence for congruence subgroups, even in the lack of property (T).

Relying on intricate estimates involving the trace formula, Raimbault \cite{Raimbault} and Fraczyk \cite{Fraczyk} were able to establish Benjamini---Schramm convergence for every sequence of congruence lattices in the rank one groups $\mathrm{SL}_2(\mathbb{R})$ and $\mathrm{SL}_2(\mathbb{C})$.

Using property ($\tau$) as a replacement for property (T), Levit \cite{levit2017benjamini} established Benjamini---Schramm convergence for every sequence of congruence lattices in any higher-rank semisimple group $G$ over  local fields. Whenever lattices in $G$ are known to satisfy the congruence subgroup property   this applies  to all irreducible lattices in $G$. 

%%%%%%%%%%%%%%%%%%%%%%%%%%%%%%%%%

%%%%%%%%%%%%%%%%%%

\subsection{The IRS compactification of moduli spaces}
%We have seen above some demonstrations of to use the compactness of $\irs(G)$ in order to study properties of its special points --- the lattices. 
One may also use $\irs(G)$ in order to obtain new compactifications of certain natural spaces. 

\begin{exam}
Let $\gS$ be a closed surface of genus $\ge 2$. Every hyperbolic structure on $\gS$ corresponds to an IRS in $\PSL_2(\BR)$. Indeed one take a random point and a random tangent vector w.r.t the normalized Riemannian measure on the unit tangent bundle and consider the associated embedding of the fundamental $\pi_1(\gS)$ in $\PSL_2(\BR)$ via deck transformations.

Taking the closure in IRS$(G)$ of the set of hyperbolic structures on $\gC$, one obtains an interesting compactification of the moduli space of $\gS$.  

\begin{prob}
Analyse the IRS compactification of $\text{Mod}(\gS)$.
\end{prob} 

Note that the resulting compactification is similar to (but is not exactly) the Deligne--Munford compactification.

\end{exam}

%%%%%%%%%%%%%%%%%%%%%%%%%%%%%%%%%%%%%%%%%%%%%

\section{The Stuck--Zimmer theorem}\label{sec:SZ}

One of the most remarkable manifestations of rigidity for  invariant random subgroups is the following celebrated result due to Stuck and Zimmer \cite{SZ}.

\begin{theorem}
\label{thm:the Stuck Zimmer theorem}
Let $k$ be a local field and  $G$ be connected, simply connected semi-simple linear algebraic $k$-group. Assume that $G$ has no $k$-anisotropic factors,    has Kazhdan's property $(T)$ and $\mathrm{rank}_k(G) \ge 2$. Then every properly ergodic and irreducible probability measure preserving Borel action of $G$ is essentially free.
\end{theorem}

We recall that a probability measure preserving action is properly ergodic if it is ergodic and not essentially transitive, and is irreducible if every non-central normal subgroup is acting ergodically. 

In the work of Stuck and Zimmer   $G$ is assumed to be a   Lie group. The  modifications necessary to deal with arbitrary local fields were carried on in \cite{Levit}. Much more generally, Bader and Shalom obtained a variant of the Stuck--Zimmer theorem for products of locally compact groups with property (T) in  \cite{BSh}.

The connection between invariant random subgroups and stabilizer structure for probability measure preserving actions allows one to derive the following:

\begin{theorem}
\label{thm:IRS rigidity}
Let $G$ be as in Theorem \ref{thm:the Stuck Zimmer theorem}. Then any irreducible invariant random subgroup of $G$ is either $\delta_{\{e\}}, \delta_G $ or $\mu_\Gamma$ for some irreducible lattice $\Gamma$ in $G$.
\end{theorem}

We would like to point out that the Stuck--Zimmer theorem is a generalization of the following normal subgroup theorem of Margulis.

\begin{theorem}
\label{theorem:normal subgroup theorem}
Let $G$ be as in Theorem \ref{thm:the Stuck Zimmer theorem} and $\Gamma $  an irreducible lattice in $G$. Then any non-trivial normal subgroup $N\lhd \Gamma$ is either central or has finite index in $\Gamma$.
\end{theorem}

The Stuck--Zimmer theorem implies the normal subgroup theorem --- indeed the ideas that go into its proof build upon the ideas of Margulis. One key ingredient is the intermediate factor theorem of Nevo and Zimmer, which in turn generalizes the factor theorem of Margulis. We point out that this aspect of the proof is entirely independent of property (T). 

\begin{ques}
\label{quest:without property T}
Do Theorems \ref{thm:the Stuck Zimmer theorem} and \ref{thm:IRS rigidity} hold for all higher rank semisimple linear groups, regardless of property (T)?
\end{ques}

Observe that  the role played by  Kazhdan's property (T)  in the proof of Theorem \ref{thm:the Stuck Zimmer theorem} is in establishing the following fact.

\begin{prop}
\label{thm:property T implies not weakly amenable}
Let $G$ be a second countable locally compact group with Kazhdan's property $(T)$. Then every properly ergodic  probability measure preserving Borel action of $G$ is not weakly amenable in the sense of Zimmer.
\end{prop}

The Nevo--Zimmer intermediate factor theorem is then  used to show that a non-weakly amenable action is essentially free. On the other hand, the contrapositive of weak amenability follows quite readily from the fact that there are no non-trivial cocycles into   amenable groups associated to  probability measure preserving Borel actions of $G$.

To summarize,  it is currently unknown if  groups like $\SL_(\BR) \times \SL_2(\BR)$ or $\SL_2(\BQ_p) \times \SL_2(\BQ_p)$ admit any non-trivial irreducible invariant random subgroups not coming from  lattices.

We would like to point out that rank one semisimple linear groups,   discrete hyperbolic or relatively hyperbolic groups as well as  mapping class groups and $\mathrm{Out}(F_n)$ have a large supply of exotic invariant random subgroups \cite{Bo1, DGO, BGK}.

%%%%%%%%%%%%%%%%%%%%%

\section{The Benjamini--Schramm topology}

Let $\mathcal{M}$ be the space of all (isometry classes of) pointed proper metric spaces equipped with the Gromov--Hausdorff topology. This is a huge space and for many applications it is enough to consider compact subspaces of it obtained by bounding the geometry. That is, let $f(\gep,r)$ be an integer valued function defined on $(0,1)\times\BR^{>0}$, and let $\mathcal{M}_f$ consist of those spaces for which $\forall \gep,r$, the $\gep$-entropy of the $r$-ball 
$B_X(r,p)$ around the special point is bounded by $f(\gep,r)$, i.e. no $f(\gep,r)+1$ points in $B_X(r,p)$ form an $\gep$-discrete set. Then $\mathcal{M}_f$ is a compact subspace of $\mathcal{M}$.   

In many situations one prefers to consider some variants of $\mathcal{M}$ which carry more information about the spaces. 
For instance when considering graphs, it may be useful to add colors and orientations to the edges. The Gromov--Hausdorff distance defined on these objects should take into account the coloring and orientation.
%For instance when considering graphs, it may be useful to add colors and directions to the edges, and the distance between rooted coloured graphs remembers the coloring. 
Another example is smooth Riemannian manifolds, in which case it is better to consider framed manifolds, i.e. manifold with a chosen point and a chosen frame at the tangent space at that point. In that case, one replace the Gromov--Hausdorff topology by the ones determined by $(\gep,r)$ relations (see \cite[Section 3]{7S} for details), which remembers also the directions from the special point.

We define the {\it Benjamini--Schramm} space $\mathcal{BS}=\text{Prob}(\mathcal{M})$ to be the space of all Borel probability measures on $\mathcal{M}$ equipped with the weak-$*$ topology. Given $f$ as above, we set $\mathcal{BS}_f:=\text{Prob}(\mathcal{M}_f)$. Note that $\mathcal{BS}_f$ is compact.

The name of the space is chosen to hint that this is the same topology induced by `local convergence', considered by Benjamini and Schramm in \cite{BS}, when restricting to measures on rooted graphs. Recall that a sequence of random rooted bounded degree graphs converges to a limiting distribution iff for every $n$ the statistics of the $n$ ball around the root (i.e. the probability vector corresponding to the finitely many possibilities for $n$-balls) converges to the limit. 

The case of general proper metric spaces can be described similarly. A sequence $\mu_n\in\mathcal{BS}_f$ converges to a limit $\mu$ iff for any compact pointed `test-space' $M\in\mathcal{M}$, any $r$ and some arbitrarily small\footnote{This doesn't mean that it happens for all $\gep$.} $\gep>0$, the $\mu_n$ probability that the $r$ ball around the special point is `$\gep$-close' to $M$ tends to the $\mu$-probability of the same event.

\begin{exam}
An example of a point in $\mathcal{BS}$ is a measured metric space, i.e. a metric space with a Borel probability measure. 
A particular case is a finite volume Riemannian manifold --- in which case we scale the Riemannian measure to be one, and then randomly choose a point and a frame.
\end{exam}

Thus a finite volume locally symmetric space $M=\gC\backslash G/K$ produces both a point in the Benjamini--Schramm space and an IRS in $G$. This is a special case of a general analogy that I'll now describe. Given a symmetric space $X$, let us denote by $\mathcal{M}(X)$ the space of all pointed (or framed) complete Riemannian orbifolds whose universal cover is $X$, and by
$\mathcal{BS}(X)=\text{Prob}(\mathcal{M}(X))$ the corresponding subspace of the Benjamini--Schramm space. 

Let $G$ be a non-compact simple Lie group with maximal compact subgroup $K\le G$ and an associated Riemannian symmetric space $X=G/K$. There is a natural map 
$$
 \{\text{discrete subgroups of }~G\}\to \mathcal{M}(X),~\gC\mapsto \gC\backslash X.
$$  
It can be shown that this map is continuous, hence inducing a continuous map
$$
 \text{DIRS}(G)\to \mathcal{BS}(X).
$$
It can be shown that the later map is one to one, and since $\text{DIRS}(G)$ is compact, it is
a homeomorphism to its image (see \cite[Corollary 3.4]{7S}). \footnote{It was Miklos Abert who pointed out to me, about eight years ago, the analogy between Benjamini--Schramm convergence (at that time `local convergence') and convergence of Invariant Random Subgroups,
which later played an important rule in the work \cite{7S} and earlier in the work of Abert-Glasner-Virag \cite{AGV1,AGV2}.
}
%Thus the IRS topology coincides with the BS-topology if we identify $\irs_d(G)$ with its image.

\begin{rem}[Invariance under the geodesic flow]
Given a tangent vector $\overline{v}$ at the origin (the point corresponding to $K$) of $X=G/K$, define a map 
$\mathcal{F}_{\overline{v}}$ from $\mathcal{M}(X)$ to itself by moving the special point using the exponent of 
$\overline{v}$ and applying parallel transport to the frame. This induces a homeomorphism of $\mathcal{BS}(X)$. The image of $\text{DIRS}(G)$ under the map above is exactly the set of $\mu\in \mathcal{BS}(X)$ which are invariant under $\mathcal{F}_{\overline{v}}$ for all $\overline{v}\in T_K(G/K)$.
\end{rem}

Thus we can view geodesic-flow invariant probability measures on framed locally $X$-manifolds as IRS on $G$ and vice versa, and the Benjamini--Schramm topology on the first coincides with the IRS-topology on the second.

\begin{rem}
The analogy above can be generalised, to some extent, to the context of general locally compact groups. Given a locally compact group $G$, fixing a right invariant metric on $G$, we obtain a map $\text{Sub}_G\to\mathcal{M},~H\mapsto G/H$, where the metric on $G/H$ is the induced one. Moreover, this map is continuous hence defines a continuous map
$\text{IRS}(G)\to \mathcal{BS}$.
\end{rem}

For the sake of simplicity let us forget `the frame' and consider pointed $X$-manifolds, and $\mathcal{BS}(X)$ as probability measures on such. We note that while for general Riemannian manifolds there is a benefit for working with framed manifolds, for locally symmetric spaces of non-compact type, pointed manifolds, and measures on such, behave nicely enough.
 
In order to examine convergence in $\mathcal{BS}(X)$ it is enough to use as `test-space' balls in locally $X$-manifolds. Moreover, since $X$ is non-positively curved, a ball in an $X$-manifold is isometric to a ball in $X$ iff it is contractible. 

Note that since $X$ is a homogeneous space, all choices of a probability measure on $X$ correspond to the same point in $\mathcal{BS}(X)$. Abusing notations, we shall denote this point by $X$.

\begin{defn}
Let us say that a sequence in $\mathcal{BS}(X)$ is {\it Farber} if it converges to $X$. 
\end{defn}

For an $X$-manifold $M$ and $r>0$, we denote by $M_{\ge r}$ the $r$-thick part in $M$:
$$
 M_{\ge r}:=\{x\in M:\text{InjRad}_M(x)\ge r\},
$$
where $\text{InjRad}_M(x)=\sup\{\gep:B_M(x,\gep)~\text{is contractible}\}$. 

\begin{prop}\cite[Corollary 3.8]{7S}
A sequence $M_n$ of finite volume $X$-manifolds is Farber iff
$$
 \frac{\vol((M_n)_{\ge r})}{\vol(M_n)}\to 1,
$$
for every $r>0$.
\end{prop}

Theorem \ref{thm:7main} can be reformulated as:

\begin{thm}
Let $X$ be an irreducible Riemannian symmetric space of non compact type of rank at least $2$. For any $r$ and $\gep$ there is $V$ such that if $M$ is an $X$-manifold of volume $v\ge V$ then $\frac{\vol(M_{\ge r})}{v}\ge 1-\gep$ (see Figure \ref{fig:whale}).
\end{thm}

\begin{figure}[h]
    \centering
    \includegraphics[width=0.8\textwidth]{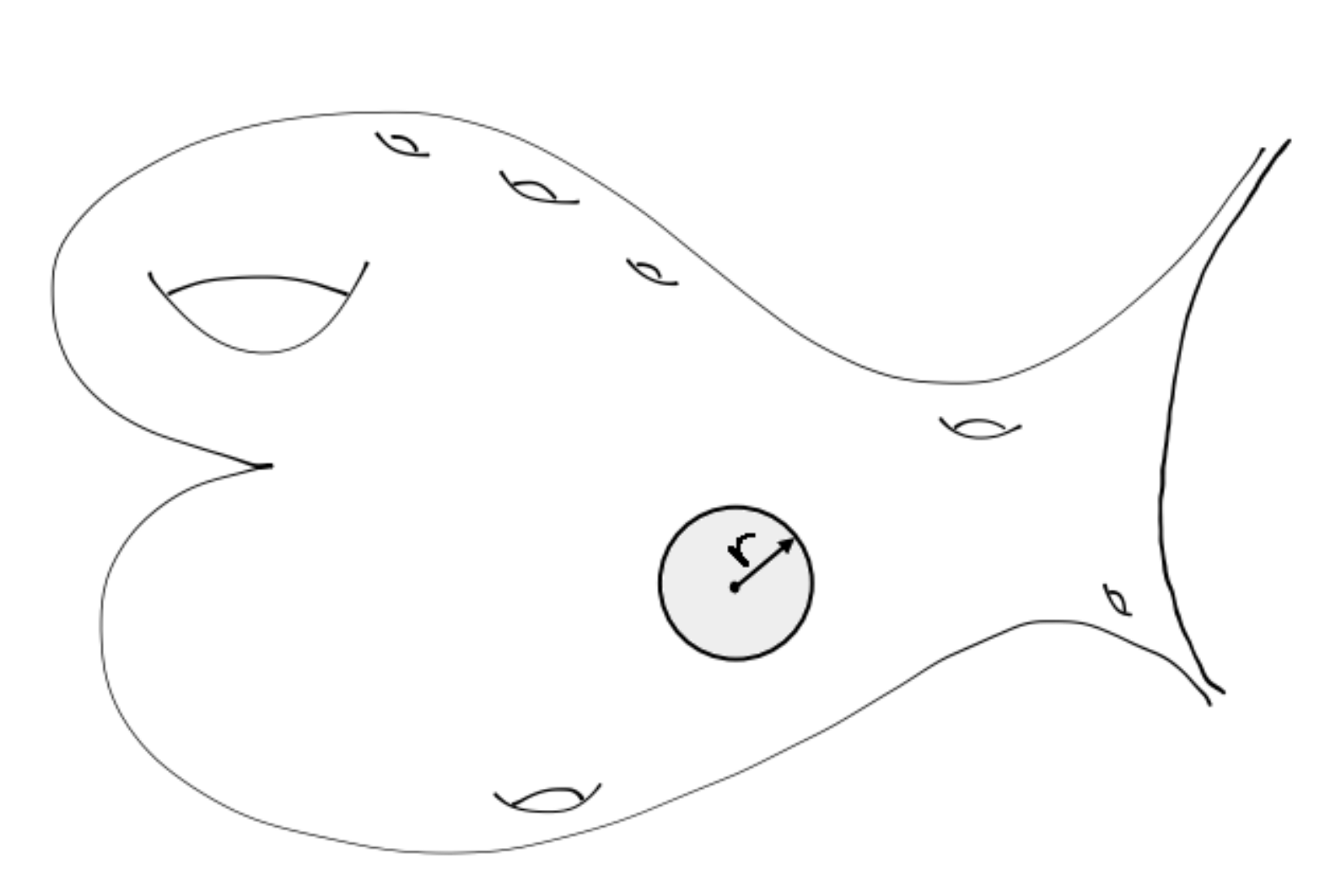}
    \caption{A large volume manifold is almost everywhere fat.}
    \label{fig:whale}
\end{figure}

%\begin{rem}
%I was introduced to the notion of Benjamini--Schramm convergence, at that time `local convergence', around 2010 by Miklos Abert. Miklos also explained to me the analogy with Invariant Random Subgroups which later played a central rule in the work \cite{7S} and earlier in the work of Abert-Glasner-Virag \cite{AGV1,AGV2}.
%\end{rem}

%%%%%%%%%%%%%%%%%%%%%%%%%%%%%%%%%

\section{Applications to $L_2$-invariants}

%One of the major achievements of \cite{7S} is a convergence formula for relative Plancherel measures with respect to a general uniformly discrete sequence of lattices. 
%To be precise, l
Let $\Gamma$ be a uniform lattice in $G$. The right quasi-regular representation $\rho_\Gamma$ of $G$ in $L^2(\Gamma \backslash G,\mu_G)$  decomposes as a direct sum of irreducible representations.
% \cite[\S 1.2.3]{gelfand1969representation}. 
Every  irreducible unitary representation $\pi$ of $G$  appears in $\rho_\Gamma$ with  finite multiplicity $m(\pi,\Gamma)$. 

\begin{definition}
\label{def:relative plancherel}
The \emph{normalized relative Plancherel measure} of $G$ with respect to $\Gamma$ is an atomic measure on the unitary dual $\widehat{G}$  given by
	$$
	\nu_\Gamma = \frac{1}{\mathrm{vol}(\Gamma \backslash G)} \sum_{\pi \in \widehat{G}} m(\pi, \Gamma) \delta_\pi.
	$$
\end{definition}

%It turns out that the following deep representation-theoretic statement follows from the geometric asymptotic property of the kind established in Theorem \ref{thm:accumulation points of invariant random subgroups}.

%of being \emph{weakly trivial}\footnote{A sequence of lattices in a locally compact group is \emph{weakly trivial} if it satisfies the obvious analogue of Definition \ref{def:weakly central} with $Z(G)$ replaced by $\{e\}$. See Definition \ref{def: a Farber sequence of invariant random subgroups} below}.

%To be precise,  
 
%In Theorem \ref{thm:BS convergence implies plancherel convergence} and Corollaries \ref{cor:BS convergence implies ptwise plancherel convergence}
% and \ref{cor:Plancharel convergence for lattices in higher rank T} we assume that   $G$ is  a non-Archimedean semisimple analytic group over a local field of zero characteristic.

The following result, extending earlier works of DeGeorge--Wallach \cite{DW}, Delorme \cite{Del} and many others, was proved in \cite{7S} for real Lie groups and then generalized to non-archimedean groups in \cite{non-archimedean}:

\begin{theorem}
%\marginpar{maybe write this theorem for general zero characterisitc (not just non-Archimedean)}
	\label{thm:BS convergence implies plancherel convergence}
Let $G$ be a  semisimple analytic group in zero characteristic. Fix a Haar measure on $G$ and let $\nu_G$ be the associated Plancherel measure on $\widehat{G}$.
Let $\Gamma_n$ be a uniformly discrete  sequence of lattices in $G$ with $\mu_{\Gamma_n}$ being weak-$*$ convergent to $\delta_{\{e\}}$.  Then 
	$$ \nu_{\Gamma_n}(E) \xrightarrow{n \to \infty} \nu_G(E) $$
	for every relatively quasi-compact  $\nu^G$-regular  subset $E \subset \widehat{G}$.
\end{theorem}

%The proof of Theorem \ref{thm:BS convergence implies plancherel convergence} is essentially the same as \cite[1.2,6.7]{7S}. It relies on the Plancherel formula and the Sauvageot principle

One of the consequence of Theorem \ref{thm:BS convergence implies plancherel convergence} is the convergence of normalized Betti numbers (cf. \cite{7-note}). Recently we were able to get rid of the co-compactness and the uniform discreteness assumptions and proved the following general version, making use of the Bowen--Elek  \cite[\S 4]{Bo3} simplicial approximation technique:

\begin{thm}\cite{ABBG}\label{ABBG}
Let $X$ be a symmetric space of non compact type of
$dim(X)\neq 3$ and $(M_n)$ is a weakly convergent sequence of finite volume $X$-manifolds.  Then for all $k$, the normalized Betti numbers $b_k(M_n)/\vol(M_n)$  converge.
\end{thm}

Here, the only three-dimensional  irreducible symmetric spaces of noncompact type are scales of $\BH^3$. In fact, the conclusion of Theorem \ref{ABBG} is  false when $X=\BH^3$. As an example, let $K \subset S^3$ be a knot such that the complement $M = S^3 \setminus K$ admits a hyperbolic metric, e.g.\ the figure-8 knot. Using meridian--longitude coordinates, let $M_n$ be obtained by Dehn filling $M$ with slope $(1,n)$;  then each $M_n$ is a homology $3$-sphere. The manifolds $M_n \to M$ geometrically, see \cite[Ch E.6]{BP92}, so the measures $\mu_{M_n}$ weakly converge to $\mu_M$ (c.f.\ \cite[Lemma 6.4]{BGS16}) and the volumes $\vol(M_n) \to \vol(M)$. However, $0=b_1(M_n) \not \to b_1(M) =1$, so the normalized Betti numbers of the sequence $M_1,M,M_2,M,\ldots$ do not converge.

\begin{corollary}\label{cor:betti}
	 Suppose that $(M_n)$  is a Farber sequence of finite volume $X$-manifolds. Then for all $k\in \BN$, we have
$b_k(M_n) / \vol(M_n) \to \beta_k^{(2)}(X).$
\end{corollary}

In the thin case, we were able to push our analytic methods far enough to give a proof for $X=\BH^d$, see \cite[Theorem 1.8]{7S}. Hence,  there is no problem in allowing $X=\BH^3$ in Corollary \ref{cor:betti}. The analog of Corollary \ref{cor:betti} for $p$-adic Bruhat Tits buildings is proved in \cite{non-archimedean}.

%%%%%%%%%%%%%%%%%%%%%%%%%%%%%%%%%

\section{Measures on the space of Riemannian manifolds}\label{sec:AB}

When $X=G/K$ is a symmetric space of noncompact type, say, the quotient of a discrete, torsion-free IRS $\Gamma$ of $G$ is a random $X$-manifold $M$.  Fixing a base point $p$ in $X$, the projection of $p$ to $\Gamma \backslash X$ is a natural base point for the quotient. So, we can regard the quotient of an IRS as a random pointed $X$-manifold. In fact, the conjugation invariance of $\Gamma$ directly corresponds to a property called \emph{unimodularity} of the random pointed X-manifold, just as IRSs of discrete groups correspond to unimodular random Schreier graphs.

In \cite{AbBi}, the Abert and Biringer study unimodular probability measures on the more general space $\mathcal M^d$ of all pointed Riemannian $d$-manifolds, equipped with the smooth topology.  One can construct such unimodular measures from finite volume $d$-manifolds, or from IRSs of continuous groups as above (see \cite[Proposition 1.9]{AbBi}). Under certain geometric assumptions like pinched negative curvature or local symmetry, they show that sequences of unimodular probability measures are precompact, in parallel with the compactness of the space of IRSs of a Lie group, see \cite[Theorems 1.10 and 1.11]{AbBi}.  They also show that unimodular measures on $\mathcal M^d$ are just those that are `compatible' with its foliated structure. Namely, $\mathcal M^d$ is almost a foliated space, where a leaf is obtained by fixing a manifold $M$ and varying the basepoint. While this foliation may be highly singular, they show in \cite[Theorem 1.6]{AbBi} that after passing to an (actually) foliated desingularization, unimodular measures are just those that are created by integrating the Riemannian measures on the leaves against some invariant transverse measure. This is a precise analogue of the hard-to-formalize statement that a unimodular random graph is a random pointed graph in which the vertices are `distributed uniformly' across each fixed graph.

%%%%%%%%%%%%%%%%%%%%%%%%%%%%%%%%%

\section{Soficity of IRS}

\begin{definition}
An IRS $\mu$ is {\it co-sofic} if it is a weak-$*$ limit in $\text{IRS}(G)$ of ones supported on lattices.
\end{definition}

The following result justify the name (cf. \cite[Lemma 16]{AGN}):

\begin{prop}
Let $F_n$ be the free group of rank $n$.
A Dirac mass $\delta_N,~N\lhd F_n$ is co-sofic iff the corresponding group $G=F_n/N$ is sofic.
\end{prop}

Given a group $G$ it is natural to ask:

\begin{ques}\label{ques:sofic}
Is every IRS in $G$ co-sofic?  
\end{ques}

In particular for $G=F_n$ this is equivalent to the Aldous--Lyons conjecture that every unimodular network (supported on rank $n$ Schreier graphs) is a limit of ones corresponding to finite Schreier graphs \cite{AL}.

Therefore it is particularly intriguing to study Question \ref{ques:sofic} for $G$, a locally compact group admitting $F_n$ as a lattice. This is the case for
$G=$$SL_2(\RR),SL_2(\QQ_p)$ and $\text{Aut}(T)$.

%%%%%%%%%%%%%%%%%%%%%%%%%%%%%%%%%

\section{Exotic IRS}

In the lack of Margulis' normal subgroup theorem there are IRS supported on non-lattices. Indeed, from a lattice $\gC\le G$ and a normal subgroup of infinite index $N\lhd \gC$ one can cook an IRS in $G$ supported on the closure of the conjugacy class $N^G$.

A more interesting example in $\SO(n,1)$ (from \cite{7B}) is obtained by choosing two compact hyperbolic manifolds $A,B$ with totally geodesic boundary, each with two components, and all four components are pairwise isometric and then glue random copies of $A,B$ along an imaginary line to obtain a random hyperbolic manifold whose fundamental group is an IRS in 
$\SO(n,1)$. If $A,B$ are chosen wisely, the random subgroup obtained is not contained in a lattice. However, all IRSs obtained that way are co-sofic. Other constructions of exotic IRS in $\SO(3,1)$ are given in \cite{7B}.

%%%%%%%%%%%%%%%%%%%%%%%%%%%%%%%%

\section{Existence}

There are many well known examples of discrete groups without nontrivial IRS, for instance $\PSL_n(\BQ)$, and also the Tarski Monsters. In \cite[\S 8]{Lecture} I asked for non-discrete examples, and in particular weather the Neretin group (of almost adthomorphisms of a regular tree) admits non-trivial IRS. Recently  Le-Boudec and Matte-Bon \cite{LeBMaB} constructed an example of a non-discrete locally compact group with no non-trivial IRS. (Note however that it is not compactly generated.)

\section{Character Rigidity}\label{sec:characters}

\begin{definition}
	\label{def:character}
	Let $\gC$ be a discrete group.
	A \emph{character} on $\Gamma$ is an irreducible positive definite complex-valued class function $\varphi : \Gamma \to \BC$ satisfying $\varphi(e) = 1$.
\end{definition}

The irreducibility of $\varphi$ simply means that it cannot be written as  a convex combination of two distinct characters. This notion was introduced by Thoma in \cite{thomaA, thomaB}. In the abelian case Definition \ref{def:character} reduces to the classical notion. 
% The space of all characters of $\Gamma$ is a Choquet simplex, which means that every normalized positive definite class function can written as a linear combination of characters in a unique way.

We will say that $\Gamma$ has \emph{Character rigidity} if  only the obvious candidates occur as characters of $\Gamma$.
The following theorem of Bekka \cite{bekka} is an outstanding example of such a result.

\begin{theorem}
\label{thm:Bekka character rigidity}
Let $\varphi$ be a character of the group $\Gamma = \SL_n(\BZ)$ for $n \ge 3$. Then  either $\varphi$ factors through an irreducible representation of some finite congruence quotient $\SL_n(\BZ/N\BZ)$ or $\varphi$ vanishes outside the center of $\Gamma$.
\end{theorem}

The connection between invariant random subgroups and characters arises from the following construction.  Let $(X,\mu)$ be a Borel probability space with an action of $\Gamma$  preserving $\mu$. Consider the following  real-valued function $\varphi : \Gamma \to \RR$ that is   associated to the action of $\Gamma$. The function $g$ is given by 
$$ \varphi(\gamma) = \mu(\mathrm{Fix}(\gamma))$$
for every $\gamma \in \Gamma$, where 
$$\mathrm{Fix}(\gamma) = \{x \in X \, : \, \gamma x = x \}.$$
For instance $\varphi(\gamma) = 1$ if  $\gamma$ lies in the kernel of the action  and $\varphi(\gamma) = 0$ if $\mu$-almost every point of $X$ is not fixed by $\gamma$. It turns out that $\varphi$ is a positive define class function satisfying $\varphi(e) = 1$.  

Let $\Gamma$ be an irreducible lattice in a higher rank semisimple linear group $G$ with property (T). It can be shown by means of induced actions that Theorem \ref{thm:IRS rigidity} holds for the lattice $\Gamma$ as well, namely any properly ergodic action of $\Gamma$ has central stabilizers.
%Any probability measure preserving action action of $\Gamma$ can be induced to that of $G$. Therefore 

We see that Theorem \ref{thm:Bekka character rigidity} in fact implies Theorem \ref{thm:the Stuck Zimmer theorem} in the special case of the particular arithmetic group $\Gamma = \SL_{n}(\BZ), n \ge 3$, which in turn implies the normal subgroup theorem of Margulis\footnote{The normal subgroup theorem for $\mathrm{SL}_n(\mathbb{Z})$ with $n \ge 3$ is in fact a much older theorem, dating back to Mennicke's work on the congruence subgroup problem \cite{mennicke}.}.  A character rigidity result is in general much stronger than invariant random subgroups rigidity --- indeed, not all characters arise in the above manner from probability measure preserving actions.

Recently Peterson \cite{peterson} has been able to vastly generalize Bekka's result, as follows.

\begin{theorem}
\label{thm:Peterson Character}
Character rigidity in the sense of Theorem \ref{thm:Bekka character rigidity} holds for any irreducible lattice in a higher rank semisimple Lie group without compact factors and with property (T).
\end{theorem}

Let us survey a few  other well-known classification results for characters of discrete groups. In his original papers Thoma studied characters of the infinite symmetric group \cite{thomaA}. 
Dudko and Medynets studied characters of the Higman---Thompson and related groups \cite{DM12}. Peterson and Thom establish character rigidity for linear groups over infinite fields or localizations of orders in number fields  \cite{PT16}, generalizing several previous results \cite{kirillov, ovchinnikov}.  

%Peterson considered characters for irreducible lattices in all higher rank semisimple Lie groups without compact factors and with property (T), thereby vastly generalizing Bekka's result, see .
%Therefore Bekka's result implies the Margulis normal subgroup theorem (Theorem \ref{theorem:normal subgroup theorem}) as well as the Stuck---Zimmer theorem for   the 
%
%For every infinite-index normal subgroup $N\nrm \Gamma$ the function $\delta_N : \Gamma \to \CC$, $\delta_N = 1_{|N}$ is a character \cite[V.7.9]{takesaki2002theory}.  

%It follows from the previous discussion moreover that Theorem \ref{thm:Bekka} implies Theorem \ref{thm:SZ for lattices with T}, namely invariant random subgroup rigidity, in this particular case.

%%%%%%%%%%%%%%%%%%%%%%%%%%%%%%%%%%%%

\section{History}\label{sec:history}
The interplay between a group theoretic and geometric viewpoints characterises the theory of IRS from its beginning. Two groundbreaking papers, Stuck-Zimmer \cite{SZ} and Aldous-Lyons \cite{AL} represent these two points of view. Zimmer's work, throughout, was deeply influenced by Mackey's virtual group philosophy which draws an analogy between the subgroups of \(G$ and its ergodic actions. When $G$ is a center free, higher rank simple Lie group, it is proved in 
\cite{SZ} that every non-essentially-free ergodic action is in fact a transitive action on the cosets of a lattice subgroup. 
%In terms of IRS they give a complete classification of ergodic IRS in such groups
%\footnote{Similar statements are proved for semisimple groups and their lattices under certain irreducibility conditions.}. 
These results can be viewed as yet another implementation of higher rank rigidity, but they also show that Mackey's analogy becomes much tighter when one considers non-essentially-free actions. 

The Aldous--Lyons paper is influenced by the geometric notion of {\it{Benjamini-Schramm}} convergence in graphs, sometimes also referred to as {\it{weak convergence}} or as {\it{convergence in local statistics}}, developed in \cite{AS},\cite{BS},\cite{BLPS}. Any finite graph\footnote{or more generally an infinite graph whose automorphism group contains a lattice.} gives rise to a random rooted graph, upon choosing the root uniformly at random.  Thus the collection of finite graphs, embeds as a discrete set, into the space of Borel probability measures on the (compact) space of rooted graphs. Random rooted graph in the \(w^{*}$-closure of this set are subject to the {\it{mass transport principal}} introduced by Banjamini and Schramm \cite{BS}: For every integrable function on the space of bi-rooted graphs
\[
 \int \sum_{x \in V(G)} f(G,o,x) d \mu([G,o]) = \int \sum_{x \in V(G)} f(G,x,o) d \mu([G,o]). 
\]
Aldous and Lyons define random unimodular graphs to be random rooted graphs subject to the mass transport principal. In \cite[Question 10.1]{AL} they ask whether every random unimodular graph is in the \(w^{*}\)-closure of the set of finite graphs. When one specialises this theory to Schreier graphs of a given finitely generated group \(\Gamma\) (more generally to the quotients of the Cayley-Abels graph of a given compactly generated group \(G\)) one obtains the theory of IRS in \(\Gamma\) or in \(G\). For probability measures on the Chabauty space of subgroups \(\sub_\Gamma\) --- the mass transport principal is equivalent to invariance under the adjoint action of the group. When \(\Gamma = F_d \) is the free group and \(N \lhd \Gamma\) is a normal subgroup the group \(\Gamma/N\) is {\it{sofic}} in the sense of Gromov and Weiss if and only if the IRS \(\delta_{N}\) is a \(w^{*}\)-limit of IRS supported on finite index subgroups. Thus the Aldous-Lyons question in the setting of Schreier graphs of \(F_d\) specializes to Gromov's question whether every group is sofic. 

In a pair of papers \cite{AGV1,AGV2}, Ab\'{e}rt, Glasner and Vir\'{a}g introduced the notion of IRS and used it to answer a long standing question in graph theory. A sequence \(\{X_n\}\) of finite, distinct $d$-regular Ramanujan graphs Benjamini--Schramm converges to the universal covering tree $T_d$. They provided a quantitative estimate for this result, for a Ramanujan graph $X$,
 \[
  \text{Pr} \{x \in X \ | \ \text{inj}_{X}(x) \le \beta \log \log (|X|) \} = O \left(\log(|X|)^{-\beta} \right), 
 \]
where $\beta = (30 \log(d-1))^{-1}$, \(\inj_X(x) = \max \{R \in \N \ | \ B_X(x,R) {\text{ is contractible}} \}\) and the probability is the uniform over the vertices of $X$. The proof combines the geometric and group theoretic viewpoints in an essential way: They start with a sequence of Ramanujan (Schreier) graphs \(\{X_n\}\). Passing if necessary to a subsequence they assume that \(X_n \rightarrow \Delta \backslash F_{d/2}\), where $\Delta$ is an IRS in $F_{d/2}$. Now the main technical result of their paper shows that the Schreier graph of an IRS has to satisfy Kesten's spectral gap theorem \(\rho(\Cay(\Gamma/\Delta,S)) \ge \rho(\Cay(\Gamma,S))\) with equality if and only if $\Delta = \langle e \rangle$ a.s. Thus the limiting object is indeed the tree. 

More generally they develop the theory of Benjamini--Shcramm limits of unimodular random graphs, as well as for $\Gamma$-Schreier graphs for arbitrary finitely generated group $\Gamma$. In this case the IRS version of Kesten's theorem reads \(\rho(\Cay(\Gamma/\Delta,S)) \ge \rho(\Cay(\Gamma,S))\), with an (a.s.) equality, iff $\Delta$ is (a.s.) amenable. In hope of reproducing this same beautiful picture for general finitely generated groups, Ab\'{e}rt, Glasner and Vir\'{a}g phrased a fundamental question that was quickly answered by Bader, Duchesne and Lecureux \cite{BDL} giving rise to the following theorem: Every amenable IRS in a group $\Gamma$ is supported on the subgroups of the amenable radical of $\Gamma$.

Independently of all of the above, Lewis Bowen in \cite{Bo2}, introduced the notion of an IRS, and of the {\it{Poisson boundary relative to an IRS}}. He used these notions to solve a long standing question in dynamics --- proving that the Furstenberg entropy spectrum of the free group is a closed interval. Let \((G,\mu)\) be a locally compact group with a Borel probability measure on it, and \((X,\nu)\) a \((G,\mu)\) space. This means that \(G \curvearrowright X\) acts on \(X\) measurably and \(\nu\) is a \(\mu\)-stationary probability measure in the sense that \(\nu = \mu * \nu\). The Furstenberg entropy of this space is 
\[ h_{\mu} (X,\nu) = \int \int - \log \frac{d \eta \circ g}{d \eta} (x) d \eta (x) d \mu (g).\]
\(\Spec(G,\mu) := \{ h_{\mu}(X,\nu) \ | \ (X,\nu) {\text{ an ergodic }} (G,\mu){\text{-space}}\}\) is called {\it{the Furstenberg entropy spectrum}} and it is bounded in the interval \([0,h_{max}(\mu)]\). The value \(0\) is obtained when the action is measure preserving, and the maximal value is always attained by the Poisson boundary \(B(G,\mu)\). The study of the entropy spectrum is tightly related to the study of factors of the Poisson boundary. Nevo and Zimmer, \cite{NZ2},  consider a restricted spectrum, that comes only form actions subject to certain mixing properties and show that this restricted spectrum $\Spec'(G)$ is finite for a centre free, higher rank semisimple Lie group $G$. This result was then used in their proof of the intermediate factor theorem, which in retrospect also validated the proof of the Stuck-Zimmer theorem \cite{NZ1,NZ3,NZ4}. 
%A softer result is Nevo's enropy gap theorem \cite{N}. Every group with Kazhdan's property (T) has an {\it{entropy gap}}, namely some interval \((0,\epsilon(G,\mu))\) which is disjoint from the entropy spectrum. This kind of entropy gap was later shown to be equivalent to property (T) for countable groups in \cite{BHT}.
Bowen's work on IRS filled in a gap in the other direction --- providing as it did many examples of stationary actions. 

Let $K \in \Sub(G)$. The Poisson boundary $B(K \backslash G,\mu)$ is the (Borel) quotient of the space of all $\mu$-random walks on $K \backslash G$ under the shift $\sigma (Kg_0,Kg_1,Kg_2, \ldots) = (Kg_1,Kg_2, \ldots)$. If $\Delta=N$ happens to be normal then one retains the Kaimanovich--Vershik description of the Poisson boundary on $G/N$ and clearly $G$ acts on this space from the left giving it the structure of a \((G,\mu)\)-space. In the more general setting introduced by Bowen, $G$ still acts on the natural bundle over $\Sub(G)$, where the fibre over $K \in \Sub(G)$ is $B(K \backslash G,\mu)$. The natural action of $G$ on the space of all walks on all these coset spaces, given by $g (Kg_1,Kg_2,\ldots) = (gKg^{-1}gg_1,gKg^{-1}gg_2,\ldots)$ clearly commutes with the shift and gives rise to a well defined action of \((G,\mu)\) on this bundle. Any choice of an IRS \(\theta \in \IRS(G)\) gives rise to a \((G,\mu)\)-stationary measure on this bundle. Now for $F_d = \langle s_1,\ldots,s_d\rangle$, $\mu = \frac{1}{2d} \left(\sum_{i =1}^{d} s_i+s_{i}^{-1}\right)$ the proof that $\Spec(F_d,\mu) = [0,h_{\max}(\mu)]$ is completed by finding a certain path $\alpha: I: \rightarrow \IRS(F_d)$ in the space of IRSs with the following properties: (i) the path starts at the trivial IRS (corresponding to the action on the Poisson boundary), (ii) it ends at an IRS giving rise to arbitrarily small entropy values and (iii) the entropy function is continuous on this path. Continuity of the entropy function is very special and so are the IRS that are chosen in order to allow for this continuity. The existence of paths on the other hand is actually general, in \cite{Bo1} Bowen proves that the collection of ergodic IRS on $F_d$ that are not supported on finite index subgroups is path connected.

Vershik \cite{Ver1,Ver2}, also indpendently, arrived at IRS from his study of the representation theory and especially the characters of $S^{\infty}_{f}$ --- the group of finitely supported permutations of a countable set. To an IRS $\mu \in \IRS(\Gamma)$ in a countable group define its {\it{Vershik character}} as follows
\[\phi_{\mu}: \Gamma \rightarrow \R_{\ge 0}, \qquad \phi_{\mu}(\gamma) = \mu \left(\{\Delta \in \Sub(\Gamma) \ | \ \gamma \in \Delta \} \right). \] 
If the IRS is realized as the stabilizer $\Gamma_x$ of a random point in a p.m.p. action $\Gamma \curvearrowright (X,\nu)$ (by \cite{AGV2}, every IRS can be realized in this fashion), the same IRS is given by $\phi_{\nu}(\gamma)= \nu(Fix(\gamma))$. Vershik also describes the GNS constuctions associated with this character. Let $R = \{(x,y)\in X \times X | y \in \Gamma x\}$ and let $\eta$ be the infinite measure on $R$ given by $\int f(x,y) \eta(x,y) = \int \sum_{y \in \Gamma x} f(x,y) d \mu(x).$ $\Gamma$ acts on $R$ via its action on the first coordinate $\gamma (x,y) = (\gamma x, y)$ and hence it acts on the Hilbert space $L_2(R,\eta)$. Let $\chi(x,y) = 1_{x=y} \in L_2(R,\eta)$ be the characteristic function of the diagonal. It is easy to verify that $phi_{\mu}(\gamma)= \langle \gamma \chi, \chi\rangle$. The defintion of the Vershik character clarified the deep connection between character rigidity in the snese of Connes and and the Stuck--Zimmer theorem.

%%%%%%%%%%%%%%%%%%%%%%%%%%%%%%%%%%
%%%%%%%%%%%%%%%%%%%%%%%%%%%%%%%%%%

\begin{acknowledgement}
%I was introduced to IRS around 2010 by Yair Glasner who also taught me to appreciate the Stuck--Zimmer paper. About the same time Miklos Abert introduced me to the notion of Benjamini--Schramm convergence (at that time `local convergence') expressing to me his vision which later evolved to the work \cite{7S}.
%\end{acknowledgement}

%\begin{acknowledgement}
Several people helped me in writing up parts of this paper: Ian Biringer with \S \ref{sec:AB}, Arie Levit with \S \ref{sec:characters} and \S \ref{sec:SZ} and Yair Glasner with \S \ref{sec:history}.
\end{acknowledgement}

\end{document}